\documentclass[11pt,a4paper]{article}

\usepackage{amsthm,amsmath}
\usepackage{natbib}
\RequirePackage[colorlinks,citecolor=blue,urlcolor=blue]{hyperref}
\usepackage{amssymb}
\usepackage{bbm}
\usepackage{graphicx}

\theoremstyle{plain}  %default 
\newtheorem{thm}{Theorem}[section] 
\newtheorem{lem}[thm]{Lemma} 
\newtheorem{prop}[thm]{Proposition} 
\newtheorem{cor}[thm]{Corollary} 

\theoremstyle{definition}

\theoremstyle{remark} 
\newtheorem*{rem}{Remark}

\defcitealias{dlmf}{DLMF}

\begin{document}

\title{Convolution roots and differentiability of isotropic positive definite functions on spheres}
\author{Johanna Ziegel\thanks{Heidelberg University, Institute of Applied Mathematics, Im Neuenheimer Feld 294, 69120 Heidelberg, Germany, tel: +49 (0) 6221 54 5716, e-mail: j.ziegel@uni-heidelberg.de}}
\date{}
\maketitle

\begin{abstract}
We prove that any isotropic positive definite function on the sphere can be written as the spherical self-convolution of an isotropic real-valued function. It is known that isotropic positive definite functions on $d$-dimensional Euclidean space admit a continuous derivative of order $[(d-1)/2]$. We show that the same holds true for isotropic positive definite functions on spheres and prove that this result is optimal for all odd dimensions. 
\end{abstract}

\section{Introduction}\label{sec:intro}

For an integer $d \in \mathbb{N}$ we denote the $d$-dimensional unit sphere by $\mathbb{S}^d = \{x \in \mathbb{R}^{d+1}\;|\; \lVert x\rVert = 1\}$, where $\lVert \cdot \rVert$ denotes the Euclidean norm on $\mathbb{R}^{d+1}$. A function $f:\mathbb{S}^d \times \mathbb{S}^d \to \mathbb{R}$ is \emph{positive definite} if
\begin{equation}\label{eq:1}
\sum_{i=1}^n \sum_{j=1}^n c_ic_j f(u_i,u_j) \ge 0
\end{equation}
for all sets of points $u_1,\dots,u_n \in \mathbb{S}^d$ and coefficients $c_1,\dots,c_n \in \mathbb{R}$. The function $f$ is \emph{isotropic} if there exists a function $\bar{f}:[0,\pi] \to \mathbb{R}$ that fulfils
\begin{equation}\label{eq:isodef}
f(u,v) = \bar{f}(\theta(u,v)) \quad \text{for all $u,v \in \mathbb{S}^d$,}
\end{equation}
where the geodesic distance on $\mathbb{S}^d$ is given by $\theta:\mathbb{S}^d \times \mathbb{S}^d \to \mathbb{R}$, $\theta(u,v) = \arccos(\langle u,v\rangle)$. Here, $\langle \cdot,\cdot \rangle$ denotes the standard scalar product on $\mathbb{R}^{d+1}$.

Isotropic positive definite functions on spheres occur in statistics as correlation functions of homogeneous random fields on spheres or of star-shaped random particles. They also have applications in approximation theory where they are used as radial basis functions for interpolating scattered data on spherical domains. Recent applications in spatial statistics can be found in \citet{Banerjee2005,HuangZhangETAL2011,HansenThorarinsdottirETAL2011}; application examples in approximation theory are given in the works of \citet{XuCheney1992,FasshauerSchumaker1998,CavorettoDeRossi2010}.

The class $\Psi_d$ consists of all continuous functions $\psi:[0,\pi] \to \mathbb{R}$ with $\psi(0)=1$, such that the isotropic function $\psi(\theta(\cdot,\cdot))$ is positive definite. The classes $\Psi_d$ are nonincreasing in $d$,
\[
\Psi_1 \supset \Psi_2 \supset \dots \supset \Psi_{\infty} = \bigcap_{d=1}^{\infty}\Psi_d
\]
with the inclusions being strict. 

We define the \textit{spherical convolution} of two functions $f,g:\mathbb{S}^d\times \mathbb{S}^d \to \mathbb{R}$ as
\begin{equation*}\label{eq:convdef}
(f \circledast g) (u,v) = \int_{\mathbb{S}^d} f(u,w)g(w,v) dw,\quad \text{for all $u,v \in \mathbb{S}^d$,}
\end{equation*}
where the integration is with respect to the $d$-dimensional Hausdorff measure on $\mathbb{S}^d$. The total measure of $\mathbb{S}^d$ is denoted by $\sigma_d = 2\pi^{(d+1)/2}/\Gamma((d+1)/2)$. 
It is easy to see that the spherical self-convolution of any symmetric $L^2$-function $f$ on $\mathbb{S}^d\times \mathbb{S}^d$ is positive definite. 

Spherical convolution has been used by \citet{Wood1995,Schreiner1997}; \linebreak \citet{EstradeIstas2010,HansenThorarinsdottirETAL2011} as a tool to construct spherical positive definite functions. It is natural to ask the reverse question: Which functions can be obtained through this construction principle? We can give the following general positive answer, which we prove in Section \ref{sec:conv}.
\begin{thm}\label{thm:convolution}
Any $\psi \in \Psi_d$ has a spherical convolution root, which can be taken to be real-valued and isotropic.
\end{thm}
The techniques used to show the convolution representation theorem have lead to the solution of a further interesting problem concerning positive definite functions on spheres. 

A positive definite function $f$ on the Euclidean space $\mathbb{R}^d$ is defined analogously to \eqref{eq:1}. The function $f$ is called radial, if $f(x,y) = \tilde{f}(\lVert x-y \rVert)$ for some function $\tilde{f}:[0,\infty) \to \mathbb{R}$. \citet{Schoenberg1938} showed that radial positive definite functions on $\mathbb{R}^d$ have a continuous derivative of order $[(d-1)/2]$, where $[c]$ denotes the greatest integer less or equal to $c$. The following theorem, which will be shown in Section \ref{sec:diff1} confirms the conjecture of \citet{Gneiting2011} that the same holds true on spheres.

\begin{thm}\label{thm:diff}
The functions in the class $\Psi_d$ admit a continuous derivative of order $[(d-1)/2]$ on the open interval $(0,\pi)$.
\end{thm}

The derivatives at the point $\vartheta =0$ can be infinite or can take finite values. We believe that the same holds true at $\vartheta = \pi$. However, we are currently not able to provide simple examples for the latter claim. The \emph{powered exponential family}
\[
\psi(\vartheta) = \exp\Big(-\Big(\frac{\vartheta}{c}\Big)^{\alpha}\Big), \quad \vartheta \in [0,\pi]
\]
with parameters $c > 0$ and $\alpha \in (0,1]$ belongs to $\Psi_{\infty}$; see \citet{Gneiting2011}. For $\alpha < 1$ the first derivative at zero is $-\infty$, whereas for $\alpha=1$ it takes the value $-1/c$. The \emph{sine power function}
\[
\psi(\vartheta) = 1 - \Big(\sin\frac{\vartheta}{2}\Big)^{\alpha}, \quad \vartheta \in [0,\pi]
\]
of \citet{SoubeyrandEnjalbertETAL2008} is a member of $\Psi_{\infty}$ for $\alpha \in [0,2]$. For $\alpha \in (0,1)$, the first derivative at zero is $-\infty$; for $\alpha = 1$, we obtain $\psi'(0)=-1/2$. If $\alpha \in (1,2]$, the derivative at zero is zero. 

In the Euclidean case it is known that Theorem \ref{thm:diff} is the best possible; see \citet{Gneiting1999}. Hence, there are radial positive definite functions on $\mathbb{R}^d$ whose derivative of order $[(d-1)/2]+1$ is not continuous. The optimality of Theorem \ref{thm:diff} for $d=1,3,5,7$ follows from the results of \citet{BeatsonzuCastellETAL2011}. 
In section \ref{sec:diff2} we introduce a turning bands operator for isotropic positive definite functions on spheres to show the optimality of Theorem \ref{thm:diff} for all odd dimensions. In even dimensions it remains an open problem. However, once the optimality can be shown for $d=2$, the turning bands operator immediately also yields the assertion in all even dimensions as well. 

The convolution representation result, Theorem \ref{thm:convolution}, also has consequences that are of interest in statistical applications. Firstly, it shows, that any isotropic covariance function on the sphere can be obtained by the L\'evy based approach to modelling star-shaped random particles introduced by \citet{HansenThorarinsdottirETAL2011}. Secondly, the proof of Theorem \ref{thm:convolution}  reveals a way to resolve the identifyability issues associated with these models. It is possible to distinguish one specific convolution root amongst all possible convolution roots of a given covariance function. This is the basis of the inference procedure described in \cite{Ziegel2011}.

\section{Convolution of isotropic functions on spheres}\label{sec:prelim}
Let $L^2(\mathbb{S}^d\times \mathbb{S}^d)$ be the space of square-integrable functions on $\mathbb{S}^d\times \mathbb{S}^d$ with the Hausdorff measure. By $\langle\cdot,\cdot\rangle_{L^2}$ and $\lVert \cdot \rVert_{L^2}$ we denote the scalar product and the norm of the Hilbert space $L^2(\mathbb{S}^d\times \mathbb{S}^d)$, respectively. We consider the closed subspace $L_{d,\mathcal{I}}^2 \subset L^2(\mathbb{S}^d\times \mathbb{S}^d)$ of functions that are isotropic as defined at \eqref{eq:isodef}. 
For $f \in L_{d,\mathcal{I}}^2$ it holds for all $d+1$-dimensional orthogonal matrices $R$ that
\[
f(Ru,Rv) = \bar{f}(\theta (Ru,Rv)) = \bar{f}(\theta(u,v)) = f(u,v),\quad \text{for all $u,v \in \mathbb{S}^d$.}
\]
This property characterises the functions in $L_{d,\mathcal{I}}^2$.
\begin{prop}\label{prop:convbound}
The convolution $f\circledast g$ of $f,g \in L_{d,\mathcal{I}}^2$ is in $L_{d,\mathcal{I}}^2$ and
\begin{equation}\label{eq:convbound}
\lVert f\circledast g\rVert_{L^2} \le \sigma_d \sup_{u,v \in \mathbb{S}^d} |(f \circledast g) (u,v)| \le \lVert f \rVert_{L^2} \lVert g \rVert_{L^2}.
\end{equation}
The convolution is bilinear, commutative and
\begin{equation}\label{eq:convnorm}
\lVert f \circledast g \rVert_{L^2}^2 = \langle f\circledast f, g\circledast g\rangle_{L^2}.
\end{equation}
\end{prop}
\begin{proof}
It is easy to check that $f \circledast g$ is isotropic. 
Furthermore, by H\"older's inequality,
\begin{align*}
|(f \circledast g) (u,v)| &\le \int_{\mathbb{S}^d} |\bar{f}(\theta(u,w))\bar{g}(\theta(w,v))| dw\\
&\le \left\{\int_{\mathbb{S}^d}\bar{f}(\theta(u,w))^2dw\right\}^{1/2}\left\{\int_{\mathbb{S}^d}\bar{g}(\theta(w,v))^2dw\right\}^{1/2}\\
& \stackrel{(\ast)}{=} \left\{\frac{1}{\sigma_d}\int_{\mathbb{S}^d}\int_{\mathbb{S}^d}\bar{f}(\theta(u,w))^2 dw du\right\}^{1/2}\left\{\frac{1}{\sigma_d}\int_{\mathbb{S}^d}\int_{\mathbb{S}^d}\bar{g}(\theta(w,v))^2dw dv\right\}^{1/2}\\
& = \frac{1}{\sigma_d}\lVert f \rVert_{L^2} \lVert g \rVert_{L^2}
\end{align*}
for $u,v \in \mathbb{S}^d$. The equality at $(\ast)$ holds true, because the integrals on the left hand side do not depend on $u$, $v$, respectively. Therefore, we obtain \eqref{eq:convbound}, and, in particular, $f \circledast g \in L^2(\mathbb{S}^d\times \mathbb{S}^d)$. Bilinearity and commutativity are clear, and equation \eqref{eq:convnorm} is an application of Fubini's theorem.
\end{proof}

\citet{Schoenberg1942} characterised the functions of the classes $\Psi_d$ using Gegenbauer (or ultraspherical) polynomials. 
Let $\lambda > 0$. The Gegenbauer polynomials $C_n^{\lambda}$ for $n \in \mathbb{N}_0$ are defined by the expansion
\[
\frac{1}{(1+r^2 - 2r\cos\vartheta)^{\lambda}} = \sum_{n=0}^{\infty} r^n C_n^{\lambda}(\cos\vartheta), \quad \text{for $\vartheta \in [0,\pi]$;}
\]
see \citetalias[18.12.4]{dlmf}. We will repeatedly use that
\begin{equation}\label{eq:Cnone}
C_n^{\lambda}(1) = \frac{\Gamma(n+2\lambda)}{n! \Gamma(2\lambda)}.
\end{equation}
If $\lambda = 0$ we set $C_n^0(\cos\vartheta) = \cos(n\vartheta)$ for $\vartheta \in [0,\pi]$
as in \citet{Schoenberg1942}. We need the following important property of the Gegenbauer polynomials with $\lambda = (d-1)/2$, which follows from \citet[Theorem 3.7]{Xu2005}. For $d \ge 2$, $k, n \in \mathbb{N}_0$ and $u,v \in \mathbb{S}^d$, we have
\begin{equation}\label{eq:repkernel}
\int_{\mathbb{S}^d} C_k^{(d-1)/2}(\langle u,w\rangle) C_n^{(d-1)/2}(\langle w,v \rangle) dw = \delta_{k,n} \sigma_d \frac{d-1}{2n + d-1} C^{(d-1)/2}_n(\langle u,v\rangle), 
\end{equation}
where $\delta_{k,n}$ denotes the Kronecker delta. If $\lambda = 0$ it holds that
\begin{equation*}%\label{eq:repkerneld1}
\int_{\mathbb{S}^1} C_k^0(\langle u,w\rangle) C_n^0(\langle w,v \rangle) dw = \delta_{k,n} \pi C^0_n(\langle u,v\rangle)
\end{equation*}
for $n \in \mathbb{N}_0$, $k \in \mathbb{N}$, $u,v \in \mathbb{S}^d$, and 
$\int_{\mathbb{S}^1} C_0^0(\langle u,w\rangle) C_0^0(\langle w,v \rangle) dw = 2\pi$.

\begin{prop}\label{prop:basis}
Let $d \ge 2$. The family $\mathcal{C}_d = \{E_{d,n}\}_{n \in \mathbb{N}_0}$, where $E_{d,n}:=c_{d,n}\times$ $C_n^{(d-1)/2}(\langle \cdot,\cdot\rangle)\in L_{d,\mathcal{I}}^2$ with 
\[
c_{d,n} = \sigma_d^{-1}\sqrt{\frac{2n + d-1}{(d-1)C_{n}^{(d-1)/2}(1)}}
\]
is an orthonormal basis of $L_{d,\mathcal{I}}^2$. Furthermore, for $k, n \in \mathbb{N}_0$,
\[
E_{d,k} \circledast E_{d,n} = \delta_{k,n} \bar{c}_{d,n} E_{d,n},
\]
where
\[
\bar{c}_{d,n} = \sqrt{\frac{d-1}{(2n+d-1)C_{n}^{(d-1)/2}(1)}}.
\]
\end{prop}
\begin{proof}
By \eqref{eq:repkernel}
\begin{align*}
\int_{\mathbb{S}^d}\int_{\mathbb{S}^d} C_k^{(d-1)/2}(\langle u,v\rangle) C_n^{(d-1)/2}(\langle u,v \rangle) du dv & = \delta_{k,n}\sigma_d \frac{d-1}{2n + d-1} \int_{\mathbb{S}^d} C^{(d-1)/2}_n(\langle v,v\rangle)  dv \\
&= \delta_{k,n} \sigma_d^2 \frac{d-1}{2n + d-1} C^{(d-1)/2}_n(1),
\end{align*}
hence $\mathcal{C}_d$ is an orthonormal system. It is also a Hilbert space basis, because polynomials are dense in $L^2([-1,1])$. The second assertion is a direct consequence of \eqref{eq:repkernel}.
\end{proof} 

The following Proposition complements Proposition \ref{prop:basis} and is not hard to prove.
\begin{prop}\label{prop:basisd1}
Proposition \ref{prop:basis} also holds for $d=1$ with 
\[
c_{1,n} = \begin{cases}1/(2\pi), &\text{for $n=0$,}\\
\sqrt{2}/(2\pi), & \text{for $n \ge 1$,}\end{cases}, \quad
\bar{c}_{1,n} = \begin{cases}1, & \text{for $n=0$,}\\
\sqrt{2}/2, &\text{for $n \ge 1$.}\end{cases}
\]
\end{prop}

Propositions \ref{prop:basis} and \ref{prop:basisd1} imply that, for any function $f \in L_{d,\mathcal{I}}^2$, we have
\[
f \stackrel{L^2}{=} \sum_{n \in \mathbb{N}_0} \langle f,E_{d,n}\rangle_{L^2} E_{d,n},
\]
where $\stackrel{L^2}{=}$ means that the series on the right hand side converges unconditionally in $L^2$ to the left hand side. We call the basis $\mathcal{C}_d$ the \emph{Gegenbauer basis} of $L_{d,\mathcal{I}}^2$. The coefficients $\langle f,E_{d,n}\rangle_{L^2}$ are termed the \emph{Gegenbauer coefficients} of $f$.

\begin{prop}\label{prop:convcoef}
For any $f \in L_{d,\mathcal{I}}^2$, $n \in \mathbb{N}_0$, we have
\[
f \circledast E_{d,n} = \bar{c}_{d,n}\langle f, E_{d,n} \rangle_{L^2}E_{d,n}.
\]
\end{prop}
\begin{proof}
For $N \in \mathbb{N}$ we set $f_N = \sum_{k=0}^N \langle f,E_{d,k}\rangle_{L^2} E_{d,k}$. Then $f_N$ converges to $f$ in $L^2$. We obtain
\begin{multline*}
\lVert f \circledast E_{d,n} - \bar{c}_{d,n} \langle f,E_{d,n} \rangle_{L^2} E_{d,n} \rVert_{L^2} \\
\le \lVert f \circledast E_{d,n} - f_N \circledast E_{d,n} \rVert_{L^2}+ \lVert f_N \circledast E_{d,n} - \bar{c}_{d,n}\langle f,E_{d,n} \rangle_{L^2} E_{d,n} \rVert_{L^2}.
\end{multline*}
The last summand on the right hand side is zero by the definition of $f_N$ and Proposition \ref{prop:basis}. By Proposition \ref{prop:convbound} we obtain 
\[
\lVert f \circledast E_{d,n} - f_N \circledast E_{d,n} \rVert_{L^2} =  \lVert (f - f_N) \circledast E_{d,n} \rVert_{L^2} \le \lVert f - f_N \rVert_{L^2} \lVert E_{d,n} \rVert_{L^2} \to 0,
\]
as $N \to \infty$.
\end{proof}

\begin{cor}\label{cor:convcoefs}
For any $f \in L_{d,\mathcal{I}}^2$, $n \in \mathbb{N}_0$ we have
\[
\langle f\circledast f, E_{d,n} \rangle_{L^2} = \bar{c}_{d,n} \langle f, E_{d,n} \rangle_{L^2}^2.
\]
\end{cor} 
\begin{proof}
We have
\begin{align*}
\langle f\circledast f, E_{d,n} \rangle_{L^2} &= (\bar{c}_{d,n})^{-1}\langle f\circledast f, E_{d,n} \circledast E_{d,n}\rangle_{L^2}= (\bar{c}_{d,n})^{-1}\lVert f \circledast E_{d,n}\rVert_{L^2}^2\\
&= (\bar{c}_{d,n})^{-1}\lVert \bar{c}_{d,n} \langle f,E_{d,n} \rangle_{L^2} E_{d,n} \rVert_{L^2}^2 = \bar{c}_{d,n} \langle f,E_{d,n} \rangle_{L^2}^2,
\end{align*}
where we used Propositions \ref{prop:basis} and \ref{prop:basisd1}, equation \eqref{eq:convnorm}, and Proposition \ref{prop:convcoef} in this order.
\end{proof}

The following theorem gives a necessary condition for the existence of convolution roots in $L^2_{d,\mathcal{I}}$. In the interesting special case of nonnegative Gegenbauer coefficients this condition is also sufficient.
\begin{thm}\label{thm:convrep}
If a function $f \in L_{d,\mathcal{I}}^2$ can be represented as $f = g \circledast g$ for some $g \in L_{d,\mathcal{I}}^2$ then
\begin{equation}\label{eq:convcond}
\sum_{n =0}^{\infty} (\bar{c}_{d,n})^{-1} |\langle f, E_{d,n} \rangle_{L^2}| < \infty.
\end{equation}
If \eqref{eq:convcond} holds and $\langle f, E_{d,n} \rangle_{L^2} \ge 0$ for all $n \in \mathbb{N}_0$, then there exists a $g \in L_{d,\mathcal{I}}^2$ such that $f = g \circledast g$. The coefficients of $g$ in the Gegenbauer basis can be chosen to be nonnegative.
\end{thm}
\begin{proof}
The Hilbert space $L_{d,\mathcal{I}}^2$ is isometric to the space $\ell^2$ \citep[Corollary V.4.13]{Werner2002}. Therefore $\sum_{n \in \mathbb{N}_0} a_n E_{d,n} \in L_{d,\mathcal{I}}^2$ if and only if $(a_n)_{n \in \mathbb{N}_0} \in \ell^2$, or, equivalently, $\sum_{n=0}^{\infty} a_n^2 < \infty$.  
Suppose now that $f$ is given by $f = g \circledast g$ for some $g \in L_{d,\mathcal{I}}^2$. By Corollary \ref{cor:convcoefs} we have that 
\[
\langle g, E_{d,n} \rangle_{L^2} = \pm (\bar{c}_{d,n})^{-1/2} |\langle f, E_{d,n} \rangle_{L^2}|^{1/2},
\]
hence 
\[
\sum_{n = 0}^{\infty} (\bar{c}_{d,n})^{-1} |\langle f, E_{d,n} \rangle_{L^2}| < \infty.
\]
For the reverse implication set $g = \sum_{n \in \mathbb{N}_0} (\bar{c}_{d,n})^{-1/2} \langle f, E_{d,n} \rangle_{L^2}^{1/2} E_{d,n}$. By assumption $g \in L_{d,\mathcal{I}}^2$ and by Corollary \ref{cor:convcoefs} we have for any $n \in \mathbb{N}_0$, that 
\[ \langle g\circledast g ,E_{d,n}\rangle_{L^2} = \bar{c}_{d,n}\langle g,E_{d,n}\rangle_{L^2} = \langle f,E_{d,n}\rangle_{L^2}.\]
With Parseval's equality \citep[Theorem V.4.9]{Werner2002} this yields the claim.
\end{proof}

We conclude this section with a proposition that shows that convolution products can be uniformly approximated with respect to the Gegenbauer basis $\mathcal{C}_d$.
\begin{prop}\label{prop:unifconv}
If $f \in L_{d,\mathcal{I}}^2$ is given by $f = g\circledast g$ for some $g \in L_{d,\mathcal{I}}^2$, then for every permutation $\sigma: \mathbb{N} \to \mathbb{N}$, the sequence $(f_N)_{N \in \mathbb{N}}$ with $f_N = \sum_{k=0}^N \langle f,E_{d,\sigma(k)}\rangle_{L^2} E_{d,\sigma(k)}$ converges uniformly to $f$.
\end{prop}
\begin{proof}
Let $g_N=\sum_{k=0}^N \langle g,E_{d,\sigma(k)}\rangle_{L^2}E_{d,\sigma(k)}$. By Corollary \ref{cor:convcoefs} and Proposition \ref{prop:convcoef} we have 
\begin{align*}
f - f_N &= g \circledast g - \sum_{k=0}^N \bar{c}_{d,\sigma(k)} \langle g,E_{d,\sigma(k)} \rangle_{L^2}^2 E_{d,\sigma(k)}\\
&= g \circledast g - \sum_{k=0}^N \langle g,E_{d,\sigma(k)} \rangle_{L^2} g \circledast E_{d,\sigma(k)} = g \circledast g - g \circledast g_N = g \circledast (g-g_N).
\end{align*}
Now, we can apply Proposition \ref{prop:convbound} to the last term and use the unconditional $L^2$-convergence of $g_N$ to $g$ in order to obtain the claim.
\end{proof}

\section{Convolution roots}\label{sec:conv}

Schoenberg's characterisation of the classes $\Psi_d$ is summarised in the following theorem; cf.~\citet{Schoenberg1942}.
\begin{thm}[Schoenberg]\label{thm:Schoenberg}
The class $\Psi_d$ consists of all functions of the form
\begin{equation*}
\psi(\vartheta) = \sum_{n=0}^{\infty} b_{d,n} \frac{C_n^{(d-1)/2}(\cos\vartheta)}{C_n^{(d-1)/2}(1)}, \quad \text{for $\vartheta\in [0,\pi]$,}
\end{equation*}
with nonnegative coefficients $b_{d,n}$, such that $\sum_{n=0}^{\infty}b_{d,n} = 1$. If $d=1$, then
\begin{equation}\label{eq:schoenone}
b_{1,0} = \frac{1}{\pi}\int_0^{\pi}\psi(\vartheta)d\vartheta, \quad \text{and}\quad b_{1,n} = \frac{2}{\pi}\int_0^{\pi}\cos(n\vartheta)\psi(\vartheta)d\vartheta, \quad \text{for $n \ge 1$.}
\end{equation}
If $d \ge 2$, then for $n \in \mathbb{N}_0$
\begin{equation}\label{eq:schoentwo}
b_{d,n} = \frac{2n+d-1}{2^{3-d}\pi}\frac{\left(\Gamma(\frac{d-1}{2})\right)^2}{\Gamma(d-1)}\int_0^{\pi}\big\{C_n^{(d-1)/2}(\cos\vartheta)\big\}(\sin\vartheta)^{d-1}\psi(\vartheta)d\vartheta.
\end{equation}
\end{thm}
For a function $\psi \in \Psi_d$, we call the associated coefficients $b_{d,n}$ as given by \eqref{eq:schoenone} or \eqref{eq:schoentwo}, respectively, the \emph{$d$-dimensional Schoenberg coefficients} of $\psi$. 

A function $\psi \in \Psi_d$ is \emph{strictly positive definite} if the inequality in \eqref{eq:1} is strict for all systems of pairwise distinct points, unless all the coefficients are zero. \citet{ChenMenegattoETAL2003} show that $\psi \in \Psi_d$ for $d \ge 2$ is strictly positive definite if and only if its Schoenberg coefficients $b_{d,n}$ are strictly positive for infinitely many even and infinitely many odd integers $n$. The corresponding result for $\Psi_{\infty}$ is was derived by \citet{Menegatto1994}. Despite recent advances \cite{Sun2005} there is no concise characterisation of the strictly positive definite functions in $\Psi_1$ in terms of non-zero Schoenberg coefficients available.

We prove the following result, which is slightly more detailed than Theorem \ref{thm:convolution}.
\begin{thm}\label{thm:3.2}
For any $\psi \in \Psi_d$ there exists a function $g \in L^2_{d,\mathcal{I}}$, such that 
\[
\psi(\theta(u,v)) = (g \circledast g)(u,v), \quad \text{for all $u,v \in \mathbb{S}^d$,}
\]
and $g$ has nonnegative Gegenbauer coefficients. 
\end{thm}
\begin{proof}
First, let $d \ge 2$, $\psi \in \Psi_d$. The nonnegative Schoenberg coefficients of $\psi$ are connected to the Gegenbauer coefficients of $\psi(\theta(\cdot,\cdot))$ via 
\begin{align*}
b_{d,n} &= \frac{2n + d - 1}{2^{3-d}\pi} \frac{(\Gamma(\frac{d-1}{2}))^2}{\Gamma(d-1)} \int_0^{\pi} C_n^{(d-1)/2}(\cos\vartheta)(\sin\vartheta)^{d-1} \psi(\vartheta) d\vartheta\\
&=\frac{2n + d - 1}{2^{3-d}\pi} \frac{(\Gamma(\frac{d-1}{2}))^2}{\Gamma(d-1)} \left(2\pi\sigma_d\prod_{k=2}^{d-1}\int_0^{\pi}(\sin\vartheta)^{k-1}d\vartheta\right)^{-1}\\ & \qquad\qquad \times \int_{\mathbb{S}^d \times \mathbb{S}^d} C_n^{(d-1)/2}(\langle u,v \rangle)  \psi(\theta(u,v)) du dv\\
&= \frac{(\Gamma(\frac{d-1}{2}))^2\Gamma(\frac{d}{2})(d-1)}{\Gamma(d-1)2^{4-d}\pi^{(d+1)/2}} (\bar{c}_{d,n})^{-1} \langle E_{d,n}, \psi(\theta(\cdot,\cdot))\rangle_{L^2}.
\end{align*}
The quotient in the previous line is positive and only depends on $d$. We denote it by $\alpha_d$. In particular, $\langle E_{d,n}, \psi(\theta(\cdot,\cdot))\rangle_{L^2} \ge 0$ for all $n \in \mathbb{N}_0$.
We have
\[
\frac{C_n^{(d-1)/2}(\langle \cdot, \cdot \rangle)}{C_{n}^{(d-1)/2}(1)} = \sigma_d \bar{c}_{d,n} E_{d,n},
\]
hence
\[
\psi(\theta(\cdot,\cdot)) = \alpha_d \sigma_d \sum_{n=0}^{\infty} \langle E_{d,n}, \psi(\theta(\cdot,\cdot))\rangle_{L^2} E_{d,n}.
\]
By Theorem \ref{thm:Schoenberg} 
\[
1 = \sum_{n=0}^{\infty} b_{n,d} = \alpha_d \sum_{n=0}^{\infty} (\bar{c}_{d,n})^{-1} \langle E_{d,n}, \psi(\theta(\cdot,\cdot))\rangle_{L^2} = \alpha_d \sum_{n=0}^{\infty} (\bar{c}_{d,n})^{-1} |\langle E_{d,n}, \psi(\theta(\cdot,\cdot))\rangle_{L^2}|,
\]
hence Theorem \ref{thm:convrep} yields the claim. For $d=1$ we have
\[
b_{1,n} = \begin{cases}1/(2\pi)\langle E_{1,n},\psi(\theta(\cdot,\cdot))\rangle_{L^2}, &\text{if $n=0$,}\\
\sqrt{2}/(2\pi)\langle E_{1,n},\psi(\theta(\cdot,\cdot))\rangle_{L^2}, & \text{if $n \ge 1$,}\end{cases}
\]
hence we can apply the same arguments as above.
\end{proof}

\begin{rem} For a function $\psi \in \Psi_{d+k} \subset \Psi_{d}$ for some $k \ge 1$, Theorem \ref{thm:3.2} yields spherical convolution roots $g_{d+k} \in L^2_{d+k,\mathcal{I}}$ and $g_{d} \in L^2_{d,\mathcal{I}}$ with respect to the convolution in $\mathbb{S}^{d+k}$ and $\mathbb{S}^d$, respectively. The associated functions $\bar{g}_{d+k}$, $\bar{g}_{d}$ are both defined on $[0,\pi]$ and one would hope for a simple functional relationship between them, but it remains elusive thus far. However, on the level of Schoenberg coefficients, the functions $g_{d+2}$ and $g_d$ are easily put in relation using \citet[Corollary 3]{Gneiting2011}.
\end{rem}

Let $\psi \in \Psi_d$. The construction in the proofs of Theorems \ref{thm:convrep} and \ref{thm:3.2} shows that the class $\mathcal{G}_d(\psi)$ of all spherical convolution roots $g \in L_{d,\mathcal{I}}^2$ of $\psi$ is given by all functions $g \in L_{d,\mathcal{I}}^2$, whose Gegenbauer coefficients are given by 
\begin{equation}\label{eq:sig}
\big(\alpha_d^{-1/2}\sigma_n b_{d,n}^{1/2}\big)_{n \in \mathbb{N}_0},
\end{equation}
where $(b_{d,n})_{n \in \mathbb{N}_0}$ are the Schoenberg coefficients of $\psi$ and $(\sigma_n)_{n\in\mathbb{N}_0}$ is a sequence with $\sigma_n \in \{-1,1\}$; cf.~Figure \ref{fig:1}. In Theorem \ref{thm:3.2} we identify a unique convolution root by setting $\sigma_n=1$ for all $n \in \mathbb{N}_0$. This choice resolves the identifyability issue when inferring the kernel of L\'evy based models for star-shaped random particles from their covariance or correlation structure as mentioned in Section \ref{sec:intro}. See also \citet{HansenThorarinsdottirETAL2011,Ziegel2011}.

We conclude the section by using the convolution representation to calculate the Schoenberg coefficients of the function 
\[
\iota_d: [0,\pi] \to \mathbb{R}, \vartheta \mapsto \frac{1}{\nu_d(r)}\overline{\mathbbm{1}\{\theta(\cdot,\cdot) \le r\} \circledast \mathbbm{1}\{\theta(\cdot,\cdot) \le r\}}(\vartheta),
\]
where $r \in (0,\pi/2]$, and $\nu_d$ is the normalising constant ensuring that $\iota_d(0) = 1$. The convolution is taken in $\mathbb{S}^d \times \mathbb{S}^d$. It is a short calculation to show that $\nu_1(r) = 2r$. For $d\ge 2$ the normalising constant is given by
\begin{equation}\label{eq:nud}
%\nu_d = \frac{2\pi^{d/2}}{\Gamma(\frac{d}{2})} \int_0^r (\sin \theta)^{d-1} d\theta,
\nu_d(r) = \sigma_{d-1} \int_0^r (\sin \vartheta)^{d-1} d\vartheta.
\end{equation}
%or by the recursive formula
%\begin{equation*}
%\nu_d(r) = \sqrt{\pi}\frac{(d-2)}{(d-1)}\frac{\Gamma(\frac{d-1}{2})}{\Gamma(\frac{d}{2})} \nu_{d-1}(r) - \frac{\sigma_{d-1}}{(d-1)} \cos r (\sin r)^{d-1}.
%\end{equation*}
The function $\iota_2$ has been calculated explicitly by \citet{TovchigrechkoVakser2001}. \citet{EstradeIstas2010} provide a recursive formula for the functions $\iota_d$, $d \ge 2$.
\begin{lem}\label{lem:3.3}
Let $r \in (0,\pi/2]$. The function $\mathbbm{1}\{\theta(\cdot,\cdot) \le r\} \in L_{d,\mathcal{I}}^2$ has Gegenbauer coefficients $\{\omega_{d,n}\}_{n \in \mathbb{N}_0}$ given, for $n \ge 1$, by 
\[
\omega_{d,n} = c_{d,n} \sigma_d\sigma_{d-1} \frac{d-1}{n(n+d-1)} (\sin(r))^d C_{n-1}^{(d+1)/2}(\cos(r)),\quad \text{for $d \ge 2$},
\] 
and $\omega_{1,n} = (2\sqrt{2}/n) \sin(nr)$. Finally, $\omega_{d,0}=\nu_d(r)$, where $\nu_d(r)$ is given at \eqref{eq:nud}.
\end{lem}
\begin{proof} Suppose first that $d \ge 2$. We have
\begin{align*}
\langle\mathbbm{1}&\{\theta(\cdot,\cdot) \le r\},E_{d,n} \rangle_{L_d^2}  = c_{d,n}\int_{\mathbb{S}^d} \int_{\mathbb{S}^d}\mathbbm{1}\{\theta(u,v) \le r\} C_{n}^{(d-1)/2}(\langle u,v\rangle) du dv\\
&= c_{d,n}\sigma_d^2 \left(\int_0^\pi (\sin\vartheta)^{d-1} d\vartheta\right)^{-1}\int_0^{\pi} \mathbbm{1}\{\vartheta \le r\}C_n^{(d-1)/2}(\cos\vartheta)(\sin\vartheta)^{d-1} d\vartheta\\
&= c_{d,n}\sigma_d \sigma_{d-1} \int_{\cos(r)}^1 C_n^{(d-1)/2}(u) (1-u^2)^{(d-2)/2} du.
\end{align*}
Using $c_{d,0} = \sigma_d^{-1}$, the formula for $n=0$ follows.
By \citetalias[18.9.20]{dlmf} we have for $n \ge 1$
\begin{equation}\label{eq:Cnder2}
\frac{d}{dx} \left((1-x^2)^{d/2}C_{n-1}^{(d+1)/2}(x)\right) = - \frac{n(n+d-1)}{d-1}(1-x^2)^{(d-2)/2}C_{n}^{(d-1)/2}(x),
\end{equation}
which implies the lemma. The case $d=1$ is a simple calculation.
\end{proof}

Using the relation between the Gegenbauer and the Schoenberg coefficients calculated in the proof of Theorem \ref{thm:3.2} we obtain the following corollary.
\begin{cor}\label{cor:iotacoef}
The function $\iota_d$ is in $\Psi_d$. For $d\ge 2$ its Schoenberg coefficients are given by
\[
b_{d,0} = \frac{\nu_d(r)}{\sigma_d^2} \frac{\Gamma(\frac{d-1}{2})^2\Gamma(\frac{d}{2})(d-1)}{\Gamma(d-1)2^{4-d}\pi^{(d+1)/2}},
\]
and, for $n \ge 1$,
\[
b_{d,n} = \gamma_d(r)(2n+d-1)C_n^{(d-1)/2}(1)\left(\frac{C_{n-1}^{(d+1)/2}(\cos r )}{C_{n-1}^{(d+1)/2}(1)}\right)^2,
\]
where
\[
\gamma_{d}(r) = \frac{1}{\nu_d(r)}\frac{\Gamma(\frac{d-1}{2})^2 2^{d-2} \pi^{(d-1)/2}}{d^2 \Gamma(\frac{d}{2})}(\sin r )^{2d}. 
\]
For $d=1$, we have $b_{1,0} = r/(4\pi^3)$ and $b_{1,n} = \sqrt{2}\sin^2(nr)/(rn^2\pi^2)$ for $n \ge 1$.
\end{cor}

\begin{figure}
\begin{center}
\includegraphics{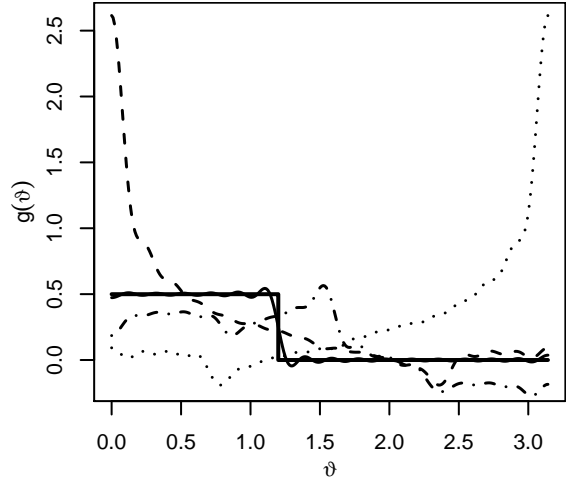}
\caption{Different convolution roots $g$ of $\iota_2(r)$ for $r=1.2$. The solid lines display the function $\nu_d(1.2)^{-1/2}\mathbbm{1}\{\vartheta \le 1.2\}$ and its approximation by the first 32 Gegenbauer polynomials. The dashed line is the convolution root with nonnegative Gegenbauer coefficients. The dotted line represents the convolution root with $\sigma_n=(-1)^n$, whereas the dash-dotted line has $\sigma_n = (-1)^{[n/2]}$, with $(\sigma_n)_{n \in \mathbb{N}_0}$ as in \eqref{eq:sig}.\label{fig:1}}
\end{center}
\end{figure}

This example illustrates that the convolution root constructed in Theorem \ref{thm:3.2} may not be the most natural one. The Gegenbauer coefficients of $\nu_d(r)^{-1/2}\mathbbm{1}\{\theta(\cdot,\cdot) \le r\}$ take both, positive and negative, signs; cf.~Lemma \ref{lem:3.3}. Hence, it is not the convolution root of $\iota_d$ that results from the construction in Theorem \ref{thm:3.2}; cf.~Figure \ref{fig:1}. The function $\iota_d$ is an example of a member of $\Psi_d$ that is supported on a spherical cap of radius $2r$. If we would like to have a convolution root that is supported on a spherical cap of radius $r$, such as $\nu_d(r)^{-1/2}\mathbbm{1}\{\theta(\cdot,\cdot) \le r\}$ for $\iota_d$, it may not be suitable to choose all coefficients of the convolution root nonnegative. In the Euclidean case, the existence of convolution roots with half-support, so-called \emph{Boas-Kac roots}, is discussed in \citet{EhmGneitingETAL2004} building on the classical result of \citet{BoasKac1945}. It remains an open problem whether Boas-Kac roots always exist for functions in $\Psi_d$.

\section{Differentiability}

\subsection{Proof of Theorem \ref{thm:diff}}\label{sec:diff1}
We denote by $\tilde{\Psi}_d$ the space of all continuous functions $\varphi:[0,\pi]\to \mathbb{R}$ which are such that the function $\varphi(\theta(\cdot,\cdot)):\mathbb{S}^d\times \mathbb{S}^d \to \mathbb{R}$ is positive definite. The difference between the spaces $\Psi_d$ and $\tilde{\Psi}_d$ is that the members $\psi \in \Psi_d \subset \tilde{\Psi}_d$ are additionally required to fulfil $\psi(0)=1$. Theorems \ref{thm:Schoenberg} and \ref{thm:3.2} also hold for the class $\tilde{\Psi}_d$ with the obvious modification that we need to require $\sum_{n=0}^{\infty} b_{d,n} < \infty$ instead of $\sum_{n=0}^{\infty} b_{d,n} =1$ for the Schoenberg coefficients in the former.

For the proof of Theorem \ref{thm:diff} on the differentiability of positive definite functions on spheres we show the following proposition, which can be applied iteratively to yield the assertion.
\begin{prop}\label{prop:der}
Let $d \ge 1$, $\psi \in \tilde{\Psi}_{d+2}$. Then $\psi$ is continuously differentiable in $(0,\pi)$ and its derivative can be written as
\[
\psi'(\vartheta) = \frac{1}{\sin\vartheta} \left(f_1(\vartheta) - f_2(\vartheta)\right),
\]
where $f_1, f_2 \in \tilde{\Psi}_d$.
%If $\psi \in \tilde{\Psi}_d$ is differentiable in $(0,\pi)$, then its derivative is given by
%\[
%\psi'(\theta) = -\sin(\theta) \frac{1}{d}\sum_{n=0}^{\infty}b_{d,n+1}(n+1)(n+d) \frac{C_n^{(d+1)/2}(\cos(\theta))}{C_n^{(d+1)/2}(1)}, \quad \theta \in (0,\pi),
%\]
%where $(b_{d,n})_{n \in \mathbb{N}}$ are the Schoenberg coefficients of $\psi$.
\end{prop}
\begin{proof} By \citetalias[18.9.19]{dlmf} the derivative of $C_n^{\alpha}$ for $\alpha > 0$ and $n \ge 1$ is given by
\begin{equation}\label{eq:Cnder}
\frac{d}{dx} C_n^{\alpha}(x) = 2\alpha C_{n-1}^{\alpha+1}(x).
\end{equation}
We assume first that $d \ge 2$. As $\tilde{\Psi}_d \supset \tilde{\Psi}_{d+2}$ we can write $\psi$ as 
\[
\psi(\vartheta) = \sum_{n=0}^{\infty} b_{d,n}\frac{C_{n}^{(d-1)/2}(\cos\vartheta)}{C_{n}^{(d-1)/2}(1)}, \quad \vartheta \in [0,\pi],
\]
with non-negative coefficients $b_{d,n}$ such that $\sum_{n=0}^{\infty}b_{d,n} < \infty$; see Theorem \ref{thm:Schoenberg}. For $N \in \mathbb{N}$, $\vartheta \in [0,\pi]$ we define
\[
\psi_N(\vartheta) = \sum_{n=0}^N b_{d,n}\frac{C_{n}^{(d-1)/2}(\cos\vartheta)}{C_{n}^{(d-1)/2}(1)}.
\]
By Proposition \ref{prop:unifconv} $\psi_N$ converges uniformly to $\psi$.
Let $\vartheta \in (0,\pi)$. By \eqref{eq:Cnder}, the derivative of $\psi_N$ is given by
\begin{align}
\psi_N'(\vartheta) &= \sum_{n=1}^N b_{d,n} (d-1) \frac{C_{n-1}^{(d+1)/2}(\cos\vartheta)}{C_{n}^{(d-1)/2}(1)} (-\sin\vartheta)\nonumber\\
&= \frac{-1}{\sin \vartheta}\sum_{n=1}^N b_{d,n} \frac{1}{C_{n}^{(d-1)/2}(1)}\left(\frac{(n+d-2)(n+d-1)}{2n+d-1}C_{n-1}^{(d-1)/2}(\cos\vartheta)\right. \nonumber\\
&\qquad \qquad\left.- \frac{n(n+1)}{2n+d-1}C_{n+1}^{(d-1)/2}(\cos\vartheta)\right)\nonumber\\
&= \frac{1}{\sin\vartheta} \sum_{n=1}^N b_{d,n} \frac{n(n+d-1)}{2n+d-1}\left(\frac{C_{n+1}^{(d-1)/2}(\cos\vartheta)}{C_{n+1}^{(d-1)/2}(1)} - \frac{C_{n-1}^{(d-1)/2}(\cos\vartheta)}{C_{n-1}^{(d-1)/2}(1)}\right),\nonumber
\end{align}
where we used \eqref{eq:Cnone}, and
\begin{multline*}%\label{eq:18.9.8}
C_n^{(d+1)/2}(\cos\vartheta)(\sin\vartheta)^2 = \frac{(n+d-1)(n+d)}{(d-1)(2n+d+1)}C_n^{(d-1)/2}(\cos\vartheta) \\- \frac{(n+1)(n+2)}{(d-1)(2n+d+1)}C_{n+2}^{(d-1)/2}(\cos\vartheta);
\end{multline*}
see \citetalias[equation (18.9.8)]{dlmf}.
Therefore
\begin{align*}
(\sin\vartheta&) \psi_N'(\vartheta)= -b_{d,1} \frac{d}{d+1} \\
&+ \sum_{n=0}^{N} \left(\frac{n(n+d-1)}{2n+d-1}b_{d,n} - \frac{(n+2)(n+d+1)}{2n+d+3}b_{d,n+2}\right)\frac{C_{n+1}^{(d-1)/2}(\cos\vartheta)}{C_{n+1}^{(d-1)/2}(1)}\\
&+ \sum_{n=N-1}^N b_{d,n+2} \frac{(n+2)(n+d+1)}{2n+d+3}\frac{C_{n+1}^{(d-1)/2}(\cos\vartheta)}{C_{n+1}^{(d-1)/2}(1)}.
\end{align*}
The last term in the above equation converges to zero uniformly in $\vartheta$ as $N \to \infty$ by \citet[Corollary 4]{Gneiting2011} and Lemma \ref{lem:Ana1}. We will omit it in the sequel. Using \citet[Corollary 3(b)]{Gneiting2011}, we obtain
\begin{multline*}
\frac{n(n+d-1)}{2n+d-1}b_{d,n} - \frac{(n+2)(n+d+1)}{2n+d+3}b_{d,n+2} 
\\= \frac{dn}{n+d}b_{d+2,n} -  \frac{d(2n+d+1)(n+2)}{(2n+d+3)(n+d)}b_{d,n+2}.
\end{multline*}
Hence,
\begin{align*}
(\sin\vartheta) \psi_N'(\vartheta)&= d\sum_{n=0}^{N}\frac{n}{n+d}b_{d+2,n}\frac{C_{n+1}^{(d-1)/2}(\cos\vartheta)}{C_{n+1}^{(d-1)/2}(1)} \\
&\quad - d\sum_{n=1}^{N+2}\frac{(2n+d-3)n}{(2n+d-1)(n+d-2)}b_{d,n}\frac{C_{n-1}^{(d-1)/2}(\cos\vartheta)}{C_{n-1}^{(d-1)/2}(1)}.
\end{align*}
We set $\beta^{(1)}_0 = 0$, 
\[
\beta^{(1)}_{n} = d\frac{n-1}{n+d-1}b_{d+2,n-1},\quad \text{for $n \ge 1$,}
\]
and 
\[
\beta^{(2)}_{n} = d\frac{(2n+d-1)(n+1)}{(2n+d+1)(n+d-1)}b_{d,n+1}, \quad\text{for $n \ge 0$.}
\]
The sequences $\{\beta^{(i)}_n\}_{n \in \mathbb{N}_0}$, $i=1,2$, are nonnegative and summable by assumption. Therefore they are the Schoenberg coefficients of some functions $f_1$, $f_2 \in \tilde{\Psi}_d$. By Proposition \ref{prop:unifconv} their partial Gegenbauer sums converge uniformly, which yields the claim. 

If $d=1$, the proof uses the same arguments with \citet[Corollary 3(a)]{Gneiting2011} instead of \citet[Corollary 3(b)]{Gneiting2011}. The Schoenberg coefficients of the functions $f_1$, $f_2$ are then given by $\beta_{n}^{(1)} = ((n-1)/n) b_{3,n-1}$, $\beta_{n}^{(2)} = b_{1,n+1}$, for $n\ge 1$, and $\beta_0^{(1)}=0$, $\beta_0^{(2)}=(1/2)b_{1,1}$.
\end{proof}

\begin{lem}\label{lem:Ana1}
Let $(\alpha_n)_{n \in \mathbb{N}}$ be an increasing sequence converging to $1$, such that the sequence $(\alpha_n^n)_{n \in \mathbb{N}}$ is bounded away from $0$. Suppose that $\sum_{n=1}^{\infty} b_n < \infty$ for some sequence $(b_n)_{n \in \mathbb{N}}$ of nonnegative numbers. If
\[
b_n \ge \alpha_n b_{n+1}, \quad \text{for all $n \in \mathbb{N}$,}
\]
then $n \, b_n \to 0$ as $n \to \infty$.
\end{lem}
\begin{proof} Let $(\alpha_n^n)_{n \in \mathbb{N}}$ be bounded below by $C > 0$.
Let $\varepsilon > 0$, choose $n_0$, such that $\sum_{k=n+1}^{m} b_k < \varepsilon$ for all $m > n > n_0$. With $m = 2n$ we obtain
\begin{multline*}
\varepsilon > \sum_{k=n+1}^{2n} b_k \ge \sum_{k=n+1}^{2n} \prod_{j=k}^{2n-1}\alpha_j b_{2n} \ge \sum_{k=n+1}^{2n}(\alpha_{n})^{2n-k} b_{2n} \\ \ge \alpha_n^{2n} n\, b_{2n} \ge C^2\, n\, b_{2n} \ge 0.
\end{multline*}
Using the same argument for $m=2n+1$ yields the claim.
\end{proof}

\subsection{Optimality of Theorem \ref{thm:diff}}\label{sec:diff2}
In this section we show that Theorem \ref{thm:diff} is optimal for all odd dimensions using similar ideas as in \citet{Gneiting1999}. We are not aware of a function $\psi \in \Psi_2$ with discontinuous derivative. If such a function was available, our method immediately also yields the optimality of the differentiability result in even dimensions. 

We introduce a \emph{turning bands operator} for isotropic positive definite functions on spheres in analogy to the Euclidean case, where the turning bands operator originates in the work of \citet{Matheron1972}.
Let $\beta = (\beta_n)_{n \in \mathbb{N}_0}$ be a sequence of real numbers. For an integer $k \in \mathbb{Z}$ we define the sequence $\beta\circ\tau_{k}$ as follows. If $k > 0$ its members are
\[
(\beta\circ\tau_{k})_n = \begin{cases} 0, &\text{if $n < k$,}\\
\beta_{n-k}, & \text{if $n \ge k$}
\end{cases}
\]
for $n \in \mathbb{N}_0$. If $k \le 0$ we put $(\beta\circ\tau_{k})_n = \beta_{n-k}$ for all $n \in \mathbb{N}_0$. Let $d \ge 1$ be an integer. For a summable sequence $\beta=(\beta_n)_{n \in \mathbb{N}}$ of nonnegative numbers $\beta_n$ we define $\psi_d(\beta,\vartheta)$ for $\vartheta \in [0,\pi]$ as 
\[
\psi_d(\beta,\vartheta) = \sum_{n=0}^{\infty} \beta_n \frac{C_n^{(d-1)/2}(\cos\vartheta)}{C_n^{(d-1)/2}(1)} \in \tilde{\Psi}_d. 
\]

\begin{prop} \label{prop:turningbands}
Let $d\ge 1$ be an integer and let $\beta = (\beta_n)_{n \in \mathbb{N}}$ be a summable sequence of nonnegative numbers $\beta_n$. Then, for all $r \in [0,\pi]$,
\begin{equation}\label{eq:turnop}
\psi_d(\beta,r) = \beta_0 + \cos r \ \psi_{d+2}(\beta \circ \tau_{-1},r) + \frac{1}{d}\sin r \ \psi_{d+2}'(\beta \circ \tau_{-1},r),
\end{equation}
and
\begin{equation}\label{eq:invop}
\frac{1}{d}(\sin r)^d \ \psi_{d+2}(\beta \circ \tau_{-1},r)= \int_0^r(\sin\vartheta)^{d-1} (\psi_d(\beta,\vartheta)-\beta_0) d\vartheta.
\end{equation}
\end{prop}
\begin{proof}
Suppose first, that $d \ge 2$.
Using Proposition \ref{prop:unifconv}, \eqref{eq:Cnder2}, and \eqref{eq:Cnone} we obtain
\begin{align*}
\int_0^r (\sin \vartheta)^{d-1}\psi_d(\beta,\vartheta) d\vartheta &= \sum_{n=0}^{\infty}\beta_n \int_0^r (\sin \vartheta)^{d-1}\frac{C_n^{(d-1)/2}(\cos\vartheta)}{C_n^{(d-1)/2}(1)}d\vartheta\\& = \beta_0\int_0^r (\sin \vartheta)^{d-1} d\vartheta + \frac{1}{d}(\sin r )^d\sum_{n=1}^{\infty} \beta_n \frac{C_{n-1}^{(d+1)/2}(\cos r)}{C_{n-1}^{(d+1)/2}(1)},\end{align*}
which implies \eqref{eq:invop}.
Differentiating both sides of \eqref{eq:invop} with respect to $r$ yields \eqref{eq:turnop}. The case $d=1$ can be shown using the same arguments.
\end{proof}

The proof of Theorem \ref{thm:diff} shows that the differentiability of a function $\psi_d(\beta,\cdot)$ only depends on the nonnegativity and the asymptotic properties of the sequence $(\beta_n)_{n \in \mathbb{N}_0}$. Therefore, for any $k \in \mathbb{Z}$ the function $\psi_d(\beta\circ\tau_k,\cdot)$ is continuously differentiable if and only if the same holds true for $\psi_d(\beta,\cdot)$. Let $c \in (0,\pi)$. Then the function
\[
\psi(\vartheta) = \max\Big\{0,\Big(1-\frac{\vartheta}{c}\Big)\Big\}, \quad \vartheta \in [0,\pi]
\]
belongs to the class $\Psi_1$ as can be shown by elementary arguments. Its first derivative does not exist at the point $\vartheta = c$. Let $\beta=(\beta_n)_{n \in \mathbb{N}_0}$ be the sequence of 1-dimensional Schoenberg coefficients of $\psi$. Let $d\ge 3$ be an odd integer. By \eqref{eq:invop} and the above remark on Theorem \ref{thm:diff}, the function $\psi_{d}(\beta \circ \tau_{-(d-1)/2},\vartheta) \in \Psi_d$ and its derivative of order $(d-1)/2$ does not exist at $\vartheta=c$.

The truncated power functions $\psi(\vartheta) = \max\{0,(1-\vartheta/c)^{\tau}\}$ were studied in detail by \citet{BeatsonzuCastellETAL2011}. They were able to show that they belong to $\Psi_d$ if $\tau \ge (d+1)/2$ for $d = 3,5,7$ and conjectured the result for all dimensions. Theorem \ref{thm:diff} immediately shows the necessity of the condition for all odd dimensions. 

\section*{Acknowledgements}
I would like to thank Tilmann Gneiting for interesting and encouraging discussions. 

\bibliographystyle{natbib}
\bibliography{biblio2}

\begin{thebibliography}{}

\bibitem[Banerjee(2005)Banerjee]{Banerjee2005}
Banerjee, S. (2005).
\newblock On geodetic distance computations in spatial modeling.
\newblock {\em Biometrics\/}, {\bfseries 61}, 617--625.

\bibitem[Beatson {\em et~al.}(2011)Beatson, zu~Castell, and
  Xu]{BeatsonzuCastellETAL2011}
Beatson, R.~K., zu~Castell, W., and Xu, Y. (2011).
\newblock A p\'olya criterion for (strict) positive definiteness on the sphere.
\newblock Preprint, \url{arXiv:1110.2437v1}.

\bibitem[Boas and Kac(1945)Boas and Kac]{BoasKac1945}
Boas, R.~P. and Kac, M. (1945).
\newblock Inequalities for {F}ourier transforms of positive functions.
\newblock {\em Duke Math. J.}, {\bfseries 12}, 189--206.
\newblock \textit{Errata} \textbf{15} (1948), 107--109.

\bibitem[Cavoretto and De Rossi(2010)Cavoretto and De Rossi]{CavorettoDeRossi2010}
Cavoretto, R. and De Rossi, A. (2010).
\newblock Fast and accurate interpolation of large scattered data sets on the
  sphere.
\newblock {\em J. Comput. Appl. Math.}, {\bfseries 234}, 1505--1521.

\bibitem[Chen {\em et~al.}(2003)Chen, Menegatto, and
  Sun]{ChenMenegattoETAL2003}
Chen, D., Menegatto, V.~A., and Sun, X. (2003).
\newblock A necessary and sufficient condition for strictly positive definite
  functions on spheres.
\newblock {\em Proc. Amer. Mat. Soc.}, {\bfseries 131}, 2733--2740.

\bibitem[{Digital Library of Mathematical Functions}(2011){Digital Library of
  Mathematical Functions}]{dlmf}
{Digital Library of Mathematical Functions} (2011).
\newblock {\em Release date 2012-03-23.}
\newblock National Institute of Standards and Technology from
  http://dlmf.nist.gov/.

\bibitem[Ehm {\em et~al.}(2004)Ehm, Gneiting, and
  Richards]{EhmGneitingETAL2004}
Ehm, W., Gneiting, T., and Richards, D. (2004).
\newblock Convolution roots of radial postitive definite functions with compact
  support.
\newblock {\em Trans. Amer. Mat. Soc.}, {\bfseries 356}, 4655--4685.

\bibitem[Estrade and Istas(2010)Estrade and Istas]{EstradeIstas2010}
Estrade, A. and Istas, J. (2010).
\newblock Ball throwing on spheres.
\newblock {\em Bernoulli\/}, {\bfseries 16}, 953--970.

\bibitem[Fasshauer and Schumaker(1998)Fasshauer and
  Schumaker]{FasshauerSchumaker1998}
Fasshauer, G.~E. and Schumaker, L.~L. (1998).
\newblock Scattered data fitting on spheres.
\newblock In M.~Daehlen, T.~Lyche, and L.~L. Schumaker, editors, {\em
  Mathematical {M}ethods for {C}urves and {S}urfaces\/}, volume~II, pages
  117--166. Vanderbilt University Press, Nashville.

\bibitem[Gneiting(1999)Gneiting]{Gneiting1999}
Gneiting, T. (1999).
\newblock On the derivatives of radial positive definite functions.
\newblock {\em J. Math. Anal. Appl.}, {\bfseries 236}, 86--93.

\bibitem[Gneiting(2012)Gneiting]{Gneiting2011}
Gneiting, T. (2012).
\newblock Strictly and non-strictly positive definite functions on spheres.
\newblock Preprint, \url{arXiv:1111.7077v4}.

\bibitem[Hansen {\em et~al.}(2011)Hansen, Thorarinsdottir, and
  Gneiting]{HansenThorarinsdottirETAL2011}
Hansen, L.~V., Thorarinsdottir, T.~L., and Gneiting, T. (2011).
\newblock L\'evy particles: {M}odelling and simulating star-shaped random sets.
\newblock {\em CSGB Research Report\/}.

\bibitem[Huang {\em et~al.}(2011)Huang, Zhang, and Robeson]{HuangZhangETAL2011}
Huang, C., Zhang, H., and Robeson, S.~M. (2011).
\newblock On the validity of commonly used covariance and variogram functions
  on the sphere.
\newblock {\em Math. Geosci.}, {\bfseries 43}, 721--733.

\bibitem[Matheron(1972)Matheron]{Matheron1972}
Matheron, G. (1972).
\newblock {Quelque Aspects de la Mont\'ee}.
\newblock Note G\'eostatistique 120, Centre de G\'eostatistique, Fontainebleau,
  France.

\bibitem[Menegatto(1994)Menegatto]{Menegatto1994}
Menegatto, V.~A. (1994).
\newblock Strictly positive definite kernels on the {H}ilbert sphere.
\newblock {\em Appl. Anal.}, {\bfseries 55}, 91--101.

\bibitem[Schoenberg(1938)Schoenberg]{Schoenberg1938}
Schoenberg, I.~J. (1938).
\newblock Metric spaces and completely monotone functions.
\newblock {\em Ann. Math.}, {\bfseries 39}, 811--841.

\bibitem[Schoenberg(1942)Schoenberg]{Schoenberg1942}
Schoenberg, I.~J. (1942).
\newblock Positive definite functions on spheres.
\newblock {\em Duke Math. J.}, {\bfseries 9}, 96--108.

\bibitem[Schreiner(1997)Schreiner]{Schreiner1997}
Schreiner, M. (1997).
\newblock Locally supported kernels for spherical spline interpolation.
\newblock {\em J. Approx. Theory\/}, {\bfseries 89}, 172--194.

\bibitem[Soubeyrand {\em et~al.}(2008)Soubeyrand, Enjalbert, and
  Sache]{SoubeyrandEnjalbertETAL2008}
Soubeyrand, S., Enjalbert, J., and Sache, I. (2008).
\newblock Accounting for roughness of circular processes: {U}sing {G}aussian
  random processes to model the anisotropic spread of airborne plant disease.
\newblock {\em Theor. Popul. Biol.}, {\bfseries 73}, 92--103.

\bibitem[Sun(2005)Sun]{Sun2005}
Sun, X. (2005).
\newblock Strictly positive definite functions on the unit circle.
\newblock {\em Math. Comp.}, {\bfseries 74}, 709--721.

\bibitem[Tovchigrechko and Vakser(2001)Tovchigrechko and
  Vakser]{TovchigrechkoVakser2001}
Tovchigrechko, A. and Vakser, I.~A. (2001).
\newblock How common is the funnel-like energy landscape in protein-protein
  interactions?
\newblock {\em Protein Science\/}, {\bfseries 10}, 1572--1583.

\bibitem[Werner(2002)Werner]{Werner2002}
Werner, D. (2002).
\newblock {\em Funktionalanalysis\/}.
\newblock Springer, Berlin, 3rd edition.

\bibitem[Wood(1995)Wood]{Wood1995}
Wood, A. T.~A. (1995).
\newblock When is a truncated covariance function on the line a covariance
  function on the circle?
\newblock {\em Stat. Probabil. Lett.}, {\bfseries 24}, 157--164.

\bibitem[Xu(2005)Xu]{Xu2005}
Xu, Y. (2005).
\newblock Lecture notes on orthogonal polynomials of several variables.
\newblock In {\em Inzell Lectures on Orthogonal Polynomials\/}, Advances in the
  Theory of Special Functions and Orthogonal Polynomials, pages 141--196. Nova
  Science Publishers, New York.

\bibitem[Xu and Cheney(1992)Xu and Cheney]{XuCheney1992}
Xu, Y. and Cheney, W. (1992).
\newblock Strictly positive definite functions on spheres.
\newblock {\em Proc. Amer. Mat. Soc.}, {\bfseries 116}, 977--981.

\bibitem[Ziegel(2012)Ziegel]{Ziegel2011}
Ziegel, J. (2012).
\newblock Stereological modelling of random particles.
\newblock {\em Comm. Statist. Theory Methods\/}.
\newblock To appear.

\end{thebibliography}

\end{document}